\documentclass[12pt,notitlepage]{amsart}
\usepackage{latexsym,amsfonts,amssymb,amsmath,amsthm}
\usepackage{graphicx}
\usepackage{color}
\usepackage[inner=1.0in,outer=1.0in,bottom=1.0in, top=1.0in]{geometry}

\newcommand{\mz}{\ensuremath{\mathbb Z}}
\newcommand{\mr}{\ensuremath{\mathbb R}}

\newcommand{\mq}{\ensuremath{\mathbb Q}}

\newcommand{\shortmod}{\ensuremath{\negthickspace \negthickspace \negthickspace \pmod}}

\newcommand{\intR}{\int_{-\infty}^{\infty}}

\newcommand{\sumstar}{\sideset{}{^*}\sum}

\theoremstyle{plain}		
	\newtheorem{mytheo}{Theorem} [section]
	
	\newtheorem{myprop}[mytheo]{Proposition}
	
     \newtheorem{mylemma}[mytheo]{Lemma}
	
	\newtheorem{myconj}[mytheo]{Conjecture}
	
	\newtheorem*{myquestion}{Question}

\theoremstyle{remark}

\numberwithin{equation}{section}

\author{Matthew P. Young} 
\address{Department of Mathematics \\
	  Texas A\&M University \\
	  College Station \\
	  TX 77843-3368 \\
		U.S.A.}
\email{myoung@math.tamu.edu}
\thanks{This material is based upon work supported by the National Science Foundation under agreement No. DMS-1101261.  Any opinions, findings and conclusions or recommendations expressed in this material are those of the author and do not necessarily reflect the views of the National Science Foundation.}
\title{The number of solutions to Mordell's equations in constrained ranges}

\begin{document}
\maketitle
 \section{Introduction}
 \subsection{Statement and main results}
The elliptic curves 
\begin{equation}
 \label{eq:Mordell}
 y^2 = x^3 + b
\end{equation}
with $b \in \mz$, $b \neq 0$, are called Mordell curves.  
 It is well-known that the number of integral points on a Mordell curve is finite.  Indeed, there exist effective bounds on the sizes of $x$ and $y$ but they are exponentially large in terms of $|b|$.  M. Hall \cite{Hall}  conjectured that $|x| \ll |b|^{2+\varepsilon}$ (which in turn implies $|y| \ll |b|^{3 + \varepsilon}$).
 The ABC conjecture implies Hall's conjecture.  See Lang's article \cite{Lang} for a survey of this story. 

 This paper was motivated by the question of counting how many integer solutions there are to $y^2 = x^3 + b$ with $|b| \leq X$.  
 This is closely related to the question of counting the number of elliptic curves over $\mq$ with discriminant bounded in absolute value by $X$.  To see the connection, note that any elliptic curve over $\mq$ may be put into Weierstrass form $y^2 = x^3 -27c_4x -54 c_6$, with $c_4, c_6 \in \mz$ and discriminant $1728\Delta = c_4^3 - c_6^2$.

 This problem of counting $D(X)$, the number of elliptic curves with $|\Delta| \leq X$,  was initiated by Brumer and McGuinness \cite{BrumerMcGuinness}, who showed that $D(X) \gg X^{5/6}$ and conjectured that $D(X) \sim c X^{5/6}$ for an explicit $c > 0$.  Fouvry, Nair, and Tenenbaum \cite{FNT} showed that  $D(X) \ll X^{1+\varepsilon}$.
 Duke and Kowalski \cite{DukeKowalski}, building on work of Brumer and Silverman \cite{BrumerSilverman}, showed that the number of isomorphism classes of elliptic cuves with conductor $q$ and $q\leq X$ is $\ll X^{1+\varepsilon}$, which also implies $D(X) \ll X^{1+\varepsilon}$ as the conductor is a divisor of the discriminant.  The number of 
 integral points on \eqref{eq:Mordell} is $\ll |b|^{\varepsilon} h_3(\mq(\sqrt{b}))$ (see \cite{BrumerSilverman}, p.99), where $h_3(K)$ denotes the $3$-part of the class number of the number field $K$.  The above-mentioned estimates of the form $D(X) \ll X^{1+\varepsilon}$ proceed by showing that the cardinality of the $3$-part is $O(|b|^{\varepsilon})$ on average.
 There are also recent non-trivial bounds on the cardinality of $h_3(\mq(\sqrt{b}))$ that hold for each $b$ \cite{Pierce} \cite{HelfgottVenkatesh} \cite{EllenbergVenkatesh}.
  It is a difficult open problem to show that $D(X) = o(X)$ (not to mention $D(X) \ll X^{5/6}$), which would have the new qualitative feature of saying that the set of integers which are the discriminant of an elliptic curve has density $0$.  
 Watkins \cite{Watkins} has given a heuristic approach to predict the asymptotics of $D(X)$ as well as the more subtle number of elliptic curves with conductor up to $X$.  
 
 It is easy to see that the number of solutions to \eqref{eq:Mordell} with $|b| \leq X$ and $|y| \ll |X|^{1/2}$ (whence $|x| \ll |X|^{1/3 }$), is $O(X^{5/6})$ (matching the expected order of magnitude predicted by Brumer and McGuinness).
 One would like to know that solutions become rarer for larger values of $|y|$, when there is some cancellation in $y^2-x^3$.
 Now suppose that $N \gg X^{1/2}$, and consider
 \begin{equation}
  T(N,X) = \# \{m \in \mz, N \leq n \leq 2N : |n^2 - m^3| \leq X\},
 \end{equation}
the number of solutions to \eqref{eq:Mordell} with $|b| \leq X$ and $N \leq y \leq 2N$.  Since we assume $N \gg X^{1/2}$, we need $n^2 \asymp m^3$ for there to be any solutions.  In other words, setting $M = N^{2/3}$, then necessarily $m \asymp M$.
We desire good upper and lower bounds on $T(N,X)$, and for example we wish to show $T(N,X) = o(X)$ for $N$ as large as possible.
 
\begin{mytheo}
\label{thm:mainthm}
 For $M^3 = N^2$, $N \gg X^{1/2}$, we have
 \begin{equation}
 \label{eq:mainthm}
  T(N,X) \ll \frac{XM}{N} +  N^{1/3+\varepsilon}. 
 \end{equation}
\end{mytheo}
The bound \eqref{eq:mainthm} shows
$T(N,X) = o(X)$ for $X \ll N \ll X^{3-\varepsilon}$.  
Hall's conjecture would imply $T(N,X) - T(N,0) = 0$ for $N \gg X^{3+\varepsilon}$ (here we subtracted $T(N,0)$ since it counts the trivial solutions $y^2 = x^3$).  
It therefore seems interesting to close the narrow gap and show that $T(N,X) - T(N,0) =o(X)$ for $X^{3-\varepsilon} \ll N \ll X^{3+\varepsilon}$ as this would imply $D(X) = o(X)$ assuming Hall's conjecture.

Actually, our main focus is an asymptotic formula for a smoothed version of $T(N,X)$.  For $i=1,2$, let $w_i$ be a fixed smooth, compactly-supported function on $\mr$, and define
the smoothed counting function
 \begin{equation}
 \label{eq:TSdef}
T_S(N,X) = \sum_{m,n \in \mz} w_1\Big(\frac{m}{M}\Big) 
w_2\big(Z(n - m^{3/2}) \big),
 \end{equation}
where
we have set 
 \begin{equation}
  Z=\frac{N}{X}.
 \end{equation}
For appropriate choices of $w_i$, we have $T(N,X) \leq T_S(N,X)$, for which we now provide a brief explanation.  We choose $w_1$ and $w_2$ nonnegative, with $w_1(m/M) = 1$ for $\eta M < m < \eta^{-1} M$ for some $\eta > 0$, and $w_2(t) = 1$ for $-1 \leq t \leq 1$.  We need to know that if $|n^2 - m^3| \leq X$ and $N \leq n \leq 2N$, then this solution is counted in \eqref{eq:TSdef}.  For this, we note that $|n-m^{3/2}| \leq \frac{X}{n+m^{3/2}} \leq \frac{1}{Z(1+ \eta^{3/2})} \leq \frac{1}{Z}$, so $w_2 = 1$ for such points.  Similarly, $N^2 - X \leq m^3 \leq X + 4N^2$, and assuming $N^2 \geq 2X$, say, we derive that $\frac12 M^3 \leq m^3 \leq \frac92 M^3$, whence $w_1(m/M) = 1$ for such points, provided $\eta < (2/9)^{1/3}$.
%
 For other choices of $w_i$, we have $T(N,X) \geq T_S(N,X)$, but we omit the details as they are quite similar to the above case.  We prefer to study the smoothed version of $T(N,X)$ because it simplifies some analytic issues, yet retains the same basic qualitative features of the un-smoothed counting function.
 
 Our main result is the following
\begin{mytheo}
\label{thm:smoothTsum}
 With weight functions $w_i$ as above, and assuming $N \gg X^{1/2}$, we have
 \begin{equation}
 \label{eq:smoothTsumsimplified}
  T_S(N,X) = \frac{XM}{N} \widehat{w_1}(0) \widehat{w_2}(0) + O(N^{1/3 + \varepsilon} + X^{1/2} N^{\varepsilon}). 
 \end{equation}
\end{mytheo}
Theorem \ref{thm:smoothTsum} implies Theorem \ref{thm:mainthm} as a short calculation shows that the $X^{1/2}$ term may be dropped since $X^{1/2} \ll X/N^{1/3} + N^{1/3}$.  
One can check that \eqref{eq:smoothTsumsimplified} is an asymptotic formula precisely for $X \gg N^{2/3 + \varepsilon}$.

In Section \ref{section:conjecture} below we speculate on the possible shape of an asymptotic formula for $T_S(N,X)$ valid for smaller values of $X$.  Our analysis connects the problem to the equidistribution of roots of quadratic congruences, as well as the theory of cubic metaplectic Eisenstein series.

The term $\frac{XM}{N} \widehat{w_1}(0) \widehat{w_2}(0)$ appearing in \eqref{eq:smoothTsumsimplified} is the (weighted) area of the region corresponding to the sum $T_S(N,X)$, which is the usual expected main term in lattice point counting problems; we shall call this the \emph{volume term}.
Note that $\frac{XM}{N} = \frac{X}{N^{1/3}}$ which naturally leads one to predict that if $N \gg X^{3+\varepsilon}$, then there are no points for such $N$ (this is one way to arrive at Hall's conjecture).  However, there is a trivial family of solutions $x=t^2, y=t^3$ showing $T(N,0) \asymp N^{1/3}$, so this line of reasoning needs to be modified to account for this.
Still, one is naturally led to wonder about the asymptotic behavior of $T_S(N,X) - T_S(N,0)$, for $X \ll N^{2/3+\varepsilon}$.  This appears to be a  challenging problem. 
One of the difficulties in obtaining more nuanced information on $T_S(N,X)$ for $X$ relatively small is that there are many one-parameter polynomial families of solutions to \eqref{eq:Mordell} with $x(t)^3 - y(t)^2 = b(t)$, with $b(t)$ of small degree.  We now briefly give a historical summary of some of these results, before returning to this discussion.

In the 1960's, Birch, Chowla, Marshall, and Schinzel \cite{BCMS} constructed rational coefficient polynomials $x(t)$, $y(t)$ with $x$ of degree $10$, $y$ of degree $15$, and $x^3-y^2$ of degree $6$.  
More generally, suppose that there exist $x, y \in \mq[t]$ where $x$ has degree $2 \delta$, $y$ has degree $3 \delta$, and $x^3 - y^2$ has degree $\delta + 1$.  Such a family is called a Davenport family, since Davenport \cite{Davenport} has shown that this is the minimal possible degree of $x^3 -y^2$.  
Beukers and Stewart \cite[Section 7.3]{BS} have given such families for $\delta = 1,2,3,4,5$; Hall \cite{Hall} gave an example with $\delta = 4$ and an example with $\delta = 3$ occurs in \cite{BCMS}.  
Elkies \cite[Section 4.1]{Elkies} has a much more comprehensive survey of some of these results.  Elkies also mentions that it is unknown if there are any Davenport families with $\delta > 5$ (if $x$ and $y$ are allowed to have complex coefficients, then such polynomials do exist, for any $\delta$).  Dujella \cite{Dujella} has given examples with $x$ and $y$ as above, but with $x^3 -y^2$ of degree $\delta + 5$, for each even $\delta = 2, 4, \dots$.  Zannier \cite{Zannier} also has some interesting discussions on these parameterized families.  As an example, the case with $\delta = 1$ is unique up to natural changes of variable, and takes the form
\begin{equation}
\label{eq:Davenportexample}
 x = t^2 + 1, \qquad y = t^3 + \tfrac{3}{2} t, \qquad x^3 - y^2 = \tfrac{3}{4} t^2 + 1.
\end{equation}

Let us see how these polynomial families constrain the size of $T(N,X)$.  If there exists a Davenport family with $x$ of degree $2 \delta$, this means $T(N,X) - T(N,0) \gg N^{\frac{1}{3 \delta}}$ provided $X \gg N^{\frac{1}{3} + \frac{1}{3 \delta}}$, with a large enough implied constant.  This follows simply by counting only the solutions given by the polynomial family.  Curiously, if $X \asymp N^{\frac{1}{3} + \frac{1}{3\delta}}$, then $\frac{XM}{N} \asymp \frac{X}{N^{1/3}} \asymp N^{\frac{1}{3 \delta}}$, so the count from this polynomial family alone is of the same order of magnitude as the volume term!  Notice also that if $X \leq \varepsilon N^{\frac{1}{3} + \frac{1}{3 \delta}}$ with a small enough $\varepsilon > 0$, then $T(N,X)$ does not count \emph{any} solutions from the polynomial family.  
Taking $\delta = 1$, and $X \asymp N^{2/3}$, we see that $T(N, X)$ is necessarily complicated in this range.

Our method of proof of Theorem \ref{thm:smoothTsum} can be easily adapted to the analog of $T_S(N,X)$ but with the additional constraint that if $d^2 | m$ and $d^3 | n$ then $d=1$.
Using M\"{o}bius inversion to detect this divisor condition, this counting function is
\begin{equation}
\label{eq:CSdef}
 C_S(N,X) = \sum_{\substack{m,n \in \mz}} w_1\Big(\frac{m}{M}\Big) 
w_2\big(Z(n - m^{3/2}) \big) \sum_{d^2 |m, \; d^3 | n} 
\mu(d).
\end{equation}
\begin{myprop}
 With conditions as in Theorem \ref{thm:smoothTsum}, we have
 \begin{equation}
  C_S(N,X) = \frac{XM}{N} \frac{\widehat{w_1}(0) \widehat{w_2}(0)}{\zeta(5)} + O(N^{1/3 + \varepsilon} + X^{1/2} N^{\varepsilon}).
 \end{equation}
\end{myprop}
\begin{proof}
In \eqref{eq:CSdef} we reverse the orders of summation and change variables $m \rightarrow d^2 m$, $n \rightarrow d^3 n$, getting
\begin{equation}
 C_S(N,X) = \sum_{d \ll N^{1/3}} \mu(d) T_S \Big(\frac{N}{d^3}, \frac{X}{d^6} \Big).
\end{equation}
Directly inserting \eqref{eq:smoothTsumsimplified} and performing trivial estimates completes the proof.
\end{proof}

\subsection{Comparison with previous work}
\label{section:literature}
Here we compare Theorem \ref{thm:mainthm} with other results in the literature.  
The estimation of $T(N,X)$ is a lattice-point counting problem for which there is an extensive literature, so it is difficult to be comprehensive.  Conveniently,
Huxley \cite{Huxley} and Trifonov \cite{Trifonov} have surveyed many of the results obtained by the known methods, and our sampling of results below is informed in large part by their summaries.  As a simple reduction step, we first note that $|y^2 - x^3| \ll X$ implies, and is implied by
\begin{equation}
\label{eq:latticecount}
 |y - x^{3/2}| \ll \frac{X}{N}.
\end{equation}
Many results in the literature concern estimates for the number of solutions to $|y-f(x)| \leq \delta$ for rather general classes of functions $f$; our technique is not applicable in general as we use special properties of $f(x) = x^{3/2}$.

The number of solutions to \eqref{eq:latticecount} (at least for $X/N$ not too small) is $c \frac{XM}{N} + O(E)$, where $E$ is an ``error'' term (it may not actually be smaller than the main term), and $c$ is a constant depending on the implied constant in \eqref{eq:latticecount}.  One classical result is $E = (MN)^{1/3} \asymp N^{5/9}$ due to Van der Corput. 

 Huxley and Trifonov \cite{HuxleyTrifonov} showed that with $M^3 = N^2$ and $N \gg X^{3/2 + \varepsilon}$, $E \ll (MN)^{3/10+\varepsilon}$, so $E\ll N^{1/2 + \varepsilon}$.  Their strategy is a generalization of a method of Swinnerton-Dyer \cite{S-D} (who dealt with $\delta = 0$ and $M = N$). 

Apparently the strongest asymptotic formula in the literature 
valid for $T(N,X)$ is due to Huxley \cite{Huxley}, who showed $E \ll (MN)^{4/15} \asymp N^{4/9}$ by a modification of a method of Bombieri-Pila \cite{BombieriPila}.  

Very recently, Baier and Browning \cite{BaierBrowning} have studied generalizations of $T(N,X)$ allowing more general cubics of the form $ay^2 + bx^3$ with $a, b \in \mz$, with a focus on the solutions restricted by the congruence $ay^2 + bx^3 \equiv 0 \pmod{q}$.  One can also view this as studying small solutions to $ay^2 + bx^3$ but where small is measured in a non-Archimedean sense.  See their Theorem 8.1 for the precise estimate.

In a different direction, Elkies \cite[Section 4.2]{Elkies} has developed an algorithm to find all the integer solutions to $|x^3-y^2| \ll M$ with $M \leq x \leq 2M$ in time $O(M^{1/2+\varepsilon})$, so in particular there are at most $O(M^{1/2+\varepsilon})$ solutions to find.  Recall that $M^{1/2} = N^{1/3}$.
In our notation, this means $T(N,N^{2/3}) \ll N^{1/3+\varepsilon}$, which agrees with our Theorem \ref{thm:mainthm} for $N \gg X^{3/2}$.  
Thus the result of Elkies is comparable to our Theorem \ref{thm:mainthm}, at least in terms of the ranges in which it can show $T(N,X) \asymp \frac{XM}{N}$; it is stronger in that it can be used to numerically find all the solutions.  Our result is complementary in that it provides a formula for a smoothed version of $T(N,X)$, with an explicit error term.  
Our method, using exponential sums, works most effectively for $X$ large (ultimately, by the duality principle in harmonic analysis), while a direct computation of $T(N,X)$ becomes less efficient with larger values of $X$, for the trivial reason that there are more solutions to count.  Elkies mentions (see \cite[final paragraph of Section 4.2]{Elkies}) that his results improved considerably on the general exponential sums techniques.  One outcome of our work here therefore puts the exponential sum method back on a roughly equal footing for this particular problem.  It is curious that Elkies also uses quite special properties of the function $f(x)=x^{3/2}$.  In particular, his lattice reduction method leads to $3$-dimensional lattices that turn out to be the symmetric-squares of $2$-dimensional lattices.  This effectively lowers the dimension in the lattice reduction step, leading to a significant gain in his approach.

We did not investigate the problem of finding an asymptotic formula for $T(N,X)$.  
The method of ``un-smoothing'' is well-known but the details can often be cumbersome, so we did not pursue this line of thought.

\subsection{Discussion of the method of proof}
 Our approach is to use the method of exponential sums.  After an application of the Poisson summation formula (in both variables), we are led to a dual exponential sum of the rough shape
 \begin{equation*}
  \sum_{0 < l \ll \frac{N}{X}} \sum_{k \asymp l \sqrt{M}} e\Big(\frac{4 k^3}{27 l^2}\Big).
 \end{equation*}
The Van der Corput approach in general converts a lattice point count into a dual exponential sum of a certain shape, but the special feature here is that we obtain a rational function as an argument of the exponential.  Furthermore, and crucially, the denominator is a \emph{square} (up to the constant factor $27$).  This extra structure is the key to additional savings which we exploit by writing $k = k_0 + 3l k_1$ where $0 \leq k_0 < 3l$ and $k_1 \asymp \sqrt{M}$.  This leads to a linear exponential sum in $k_1$ which allows for significant cancellation, and is the source of our improvements over the more general methods summarized in Section \ref{section:literature}

\subsection{Acknowledgment}
I originally encountered this problem of estimating $D(X)$ while working on my PhD thesis under Henryk Iwaniec.  It is from him that I first learned the theory of exponential sums and in particular techniques for treating sums to square moduli, and it is pleasure to thank him.  I also thank Tim Browning for interesting comments.

\section{Proof of Theorem \ref{thm:smoothTsum}}
We already mentioned how Theorem \ref{thm:smoothTsum} implies Theorem \ref{thm:mainthm}.
Rather than the bound in \eqref{eq:smoothTsumsimplified}, our method most naturally shows
\begin{equation}
 \label{eq:smoothTsum}
  T_S(N,X) = \frac{XM}{N} \widehat{w_1}(0) \widehat{w_2}(0) + O\Big(\frac{N^{2/3+\varepsilon}}{X^{1/2}} + X^{1/2} N^{\varepsilon}\Big),
 \end{equation}
 but we now argue that \eqref{eq:smoothTsumsimplified} and \eqref{eq:smoothTsum} are really equivalent.
The error term in \eqref{eq:smoothTsum} is smaller than the main term provided $X \gg N^{2/3 + \varepsilon}$, in which case \eqref{eq:smoothTsum} agrees with \eqref{eq:smoothTsumsimplified}.  In the complementary case $X \ll N^{2/3 + \varepsilon}$, \eqref{eq:smoothTsum} is really an upper bound.  We have $T_S(N,X) \leq T(N',X') \ll T_S(N',X')$ for certain choices $N' \asymp N$ and $X' \asymp X$, where the implied constants depend on the choice of weight functions $w_1, w_2$.  But then we may use the fact that $T(N, X)$ is increasing in $X$ (for fixed $N$), so in effect we may replace $X$ by $X + N^{2/3}$ in \eqref{eq:smoothTsum}, which leads to \eqref{eq:smoothTsumsimplified}.


 Our first move is common in lattice point counting problems: applying Poisson summation.  We use this in $n$ first, obtaining
 \begin{equation}
 \label{eq:TSonePoisson}
  T_S(N,X) = \sum_{m \in \mz} w_1\Big(\frac{m}{M}\big) \sum_{l \in \mz} \intR w_2(Z(y-m^{3/2})) e(-yl) dy.
 \end{equation}
Changing variables $y \rightarrow y + m^{3/2}$ and evaluating the $y$-integral gives
\begin{equation}
\label{eq:TSonePoisson2}
 T_S(N,X) = Z^{-1} \sum_{m \in \mz} w_1\Big(\frac{m}{M}\big) \sum_{l \in \mz} e(-l m^{3/2}) \widehat{w_2}\Big(\frac{l}{Z}\Big) .
\end{equation}
We chose to define $T_S(N, X)$ via \eqref{eq:TSdef} so that the variables would nicely separate here (at least, in terms of the weight functions $w_1, w_2$).
The term $l = 0$ gives the following main term, consistent with \eqref{eq:smoothTsum}:
\begin{equation}
\label{eq:Mainterm}
 Z^{-1} M \widehat{w_1}(0) \widehat{w_2}(0) + O(Z^{-1} M^{-100}).  
\end{equation}
We may as well record the effect of trivially bounding the terms with $l \neq 0$:
\begin{equation}
\label{eq:TSeasybound}
 T_S(N,X) = \frac{XM}{N} \widehat{w_1}(0) \widehat{w_2}(0) + O(M).
\end{equation}
 If $Z \ll N^{-\varepsilon}$, i.e., $X \gg N^{1+\varepsilon}$, then the sum over $l \neq 0$ is small from the rapid decay of $\widehat{w_2}$.  

From now on assume $Z \gg N^{-\varepsilon}$, and 
let $T_S'(N, X)$ denote the contribution from $l \neq 0$ in \eqref{eq:TSonePoisson}.  As we wish to exploit cancellation in the sum over $m$, we apply Poisson summation in $m$.  Thus we have
\begin{equation}
\label{eq:TS'}
 T_S'(N, X) = Z^{-1} \sum_{k \in \mz} \sum_{l \neq 0} \widehat{w_2}\Big(\frac{l}{Z}\Big)  \intR w_1\Big(\frac{x}{M}\Big) e(-l x^{3/2} + kx) dx.
\end{equation}
We shall approximate the inner $x$-integral by the stationary phase method.  
\begin{mylemma}
\label{lemma:exponentialintegral}
Let $w_1$ be a fixed smooth function on the positive reals, and let
\begin{equation}
 I =  \intR w_1\Big(\frac{x}{M}\Big) e(-l x^{3/2} + kx) dx.
\end{equation}
If $|k| \leq \eta |l| M^{1/2}$ or $|k| \geq \eta^{-1} |l| M^{1/2}$ for some sufficiently small but fixed $\eta > 0$ or if $k$ and $l$ have opposite signs, then
\begin{equation}
\label{eq:Istrongupperbound}
 I \ll (|l| \sqrt{M} + |k|)^{-100}.
\end{equation}
If $k$ and $l$ have the same sign and $k/l \asymp \sqrt{M}$, then 
\begin{equation}
\label{eq:Iasymptotic} 
 I =   c e^{ \frac{-l}{|l|} \pi i/4}
  \frac{ \sqrt{|k|}}{|l|} e\Big(\frac{4 k^3}{27l^2}\Big) w_1\Big(\frac{4 k^2}{9M l^2} \Big) + O(|l|^{-3/2} M^{-5/4}),
\end{equation}
where $c = 2 \sqrt{2}/3$.
\end{mylemma}
Remark: There are many expositions on exponential integrals in the literature, but most do not exploit the case with a $C^{\infty}$ weight function (which actually simplifies and strengthens the estimates), so for convenience we shall refer to Lemma 8.1 (a first derivative bound) and Proposition 8.2 (stationary phase) of \cite{BKY}.  We desire strong error terms at this early stage in order to clearly see where the barriers to improvement occur.  For the purposes of proving Theorems \ref{thm:mainthm} and \ref{thm:smoothTsum}, the reader could substitute the more standard estimates.

\begin{proof}[Proof of Lemma \ref{lemma:exponentialintegral}]
Suppose first that $|k| \leq \eta |l| M^{1/2}$ or $|k| \geq \eta^{-1} |l| M^{1/2}$ for some sufficiently small but fixed $\eta > 0$.  In these cases the phase $h(t) = - lt^{3/2} + kt$ satisfies $h'(t) \gg |l| M^{1/2} + |k|$.  The same estimate holds if $k$ and $l$ have opposite signs.  Then Lemma 8.1 of \cite{BKY} implies \eqref{eq:Istrongupperbound}

Next suppose $k/l \asymp \sqrt{M}$, so there is a stationary point at $x_0 = \frac{4 k^2}{9 l^2}$ and we may apply Proposition 8.2 of \cite{BKY}.  In the notation of \cite{BKY}, we have $(X,V,Y,Q) = (1, M, |l| M^{3/2}, M)$.  In \cite[(8.9)]{BKY}, $p_0(t_0)$ is the main term stated in Lemma \ref{lemma:exponentialintegral}.  For $n \geq 1$  \cite[(8.11)]{BKY} gives $p_n(t_0) \ll (|l|M^{3/2})^{-n/3}$. We use this for $n \geq 3$, and manually calculate $p_1(t_0) \ll |l|^{-1} M^{-3/2}$ using the definition of $p_1$ and the fact that $H'(t_0) = H''(t_0) = 0$, and $w''(t_0) = O(M^{-2})$.  We can also see that $p_2(t_0) \ll |l|^{-1} M^{-3/2}$, by a similar type of calculation.
Combining these estimates then leads to the stated error term in \eqref{eq:Iasymptotic} .  
\end{proof}

We return to the analysis of \eqref{eq:TS'}.  Applying Lemma \ref{lemma:exponentialintegral}, we have
\begin{equation}
\label{eq:TS'afterintegral}
 T_S'(N, X) = \sum_{\pm}   \frac{c}{Z} \sum_{k l > 0} e^{-\frac{l}{|l|} \pi i/4}  \frac{ \sqrt{|k|}}{|l|} e\Big(\frac{4 k^3}{27l^2}\Big) w_1\Big(\frac{4 k^2}{9 l^2 M} \Big) \widehat{w_2}\Big(\frac{l}{Z}\Big) + O(Z^{-1/2} M^{-3/4}).
\end{equation}
This error term $Z^{-1/2} M^{-3/4}$ is $\asymp N^{-1} X^{1/2} \ll 1$ which is more than satisfactory for \eqref{eq:smoothTsum}.
By a symmetry argument, we have
\begin{equation}
 T_S'(N,X) = 2c \text{Re}(e^{-\pi i/4} T_S''(N,X)) + O(Z^{-1/2} M^{-3/4}),
\end{equation}
where
\begin{equation}
 T_S''(N,X) = \frac{1}{Z} \sum_{l > 0} \sum_{k > 0}  \frac{ \sqrt{k}}{l} e\Big(\frac{4 k^3}{27l^2}\Big) w_1\Big(\frac{4 k^2}{9 l^2 M} \Big) \widehat{w_2}\Big(\frac{l}{Z}\Big).
\end{equation}
To track our progress so far, a trivial bound at this stage gives
\begin{equation}
\label{eq:TS''trivial}
 T_S'(N, X) \ll M^{3/4} Z^{1/2} \asymp \frac{N}{X^{1/2}}.
\end{equation}
Compared to \eqref{eq:TSeasybound}, \eqref{eq:TS''trivial} is an improvement for $N \ll X^{3/2}$.

Now we decompose the sum over $k$ as $k = k_0 + 3k_1 l$ where $0 \leq k_0  < 3l$ and $k_1 \geq 0$ (in fact, $k_1 \asymp \sqrt{M}$ from the support of $w_1$).  One easily checks $k^3 \equiv k_0^3 + 9 k_0^2 k_1 l \pmod{27 l^2}$, 
so we now have a linear exponential sum in $k_1$.   
This will often have substantial cancellation, which as usual we detect via Poisson summation.  We calculate
\begin{multline}
\label{eq:k1sumPoisson}
  \sum_{k_1 =0}^{\infty} e\Big(\frac{4k_0^2 k_1}{3l} \Big) (k_0 + 3 k_1 l)^{1/2} w_1\Big(\frac{4 (k_0 + 3k_1 l)^2}{9l^2 M} \Big) 
  \\
  = 
  \sum_{r \in \mz} \int_0^{\infty} e\Big(\frac{4k_0^2 t}{3l} \Big) (k_0 + 3 t l)^{1/2} w_1\Big(\frac{4 (k_0 + 3t l)^2}{9l^2 M} \Big)  e(-rt) dt.
 \end{multline}
After changing variables $t \rightarrow t-\frac{k_0}{3l}$, we may retain the range of integration as $0 \leq t < \infty$ because this shift is $O(1)$ while $w_1$ has support for $t \asymp \sqrt{M}$.  We may write \eqref{eq:k1sumPoisson} as
\begin{equation}
 \frac{\sqrt{3l} M^{3/4}}{2 \sqrt{2}} \sum_{r \in \mz} e\Big(\frac{-4 k_0^3}{9l^2} + \frac{rk_0}{3l} \Big) W_1\Big(\frac{\sqrt{M}}{6l} (3lr - 4 k_0^2) \Big),
\end{equation}
where $W_1$ is defined by
\begin{equation}
 W_1(y) = \int_0^{\infty} t^{1/2} w_1(t^2) e(-yt) dt.
\end{equation}
Note that $W_1$ is Schwartz-class.
Thus we obtain
 \begin{equation}
 \label{eq:TS''niceexpression}
  T_S''(N, X) = \frac{3^{1/2} M^{3/4}}{2^{3/2} Z} \sum_{l > 0} \frac{\widehat{w_2}\Big(\frac{l}{Z}\Big)}{l^{1/2}}  \sum_{0 \leq k_0 < 3 l}  \sum_{r \in \mz} e\Big(\frac{-8 k_0^3}{27l^2} + \frac{rk_0}{3l} \Big) W_1\Big(\frac{\sqrt{M}}{6l}\big(3lr -4k_0^2 \big) \Big).
 \end{equation}

 At this point it is difficult to prove any cancellation in the sum because all of the variables are constrained to essentially have $|4k_0^2 - 3lr| \ll l M^{-1/2 + \varepsilon}$.  Therefore we give up any cancellation and estimate the sum trivially.  It is natural to initially restrict to $L < l \leq 2L$ where $L$ runs over powers of $2$, with $1/2 \leq L \ll Z N^{\varepsilon}$, and then to $|4k_0^2 - 3lr| \ll \frac{L}{\sqrt{M}} N^{\varepsilon}$, using the rapid decay of $\widehat{w_2}$ and $W_1$.  We may then assume $|r| \leq \frac{4}{3} \frac{k_0^2}{l} + O(\frac{LN^{\varepsilon}}{\sqrt{M}}) \leq 24L(1+o(1))$.  Thus,
 \begin{equation}
 \label{eq:TS''expression}
  T_S''(N,X) \ll \frac{M^{3/4}}{Z} \sum_{\substack{1 \ll L \ll Z N^{\varepsilon} \\ L \text{ dyadic}}}  L^{-1/2}  U(6L, 12L, 25L, M^{-1/2} L N^{\varepsilon}) + N^{-100},
 \end{equation}
where
\begin{equation}
\label{Udef}
 U(A, B, C, D) = \# \{1 \leq |a| \leq A, |b| \leq B, |c| \leq C :  |b^2 - ac| \leq D \}.
\end{equation}
Here we wrote $b = 2k_0$ (whence $|b| \leq 6l \leq 12L$), $a =3l \leq 6L$, and $c=r$ (whence $|c| \leq 25L)$.

\begin{mylemma}
\label{lemma:Ubound}
 Suppose that $A,B, C \geq 1$, and $D \ll \min(B^2, AC)$.  Then we have
 \begin{equation}
  U(A,B,C,D) \ll A (1 +  \sqrt{D}) + B(1+D) (AC)^{\varepsilon}.
 \end{equation}
\end{mylemma}
\begin{proof}
 First, consider the elements with $c=0$.  Then we are counting $|b| \leq \min(B,  \sqrt{D} ) \ll \sqrt{D}$, and there is no condition on $a$, so the total number of such elements is
 \begin{equation}
  \ll  A (1 +  \sqrt{D}).
 \end{equation}
Next suppose $c \neq 0$.  Let $q = ac$ be a new variable, with $1 \leq |q| \ll AC$
and having multiplicity $d(q) \ll (AC)^{\varepsilon}$.  For each $b$, the number of $q$ satisfying $|b^2 - q| \leq D$ is at most $1 + 2D$, so the number of elements with $c \neq 0$ is
\begin{equation}
 \ll B(1+D) (AC)^{\varepsilon}. \qedhere
\end{equation}
\end{proof}
Applying Lemma \ref{lemma:Ubound} to \eqref{eq:TS''expression}, we have
\begin{equation}
 T_S''(N,X) \ll \frac{M^{3/4+\varepsilon}}{Z} \sum_{\substack{1 \ll L \ll Z N^{\varepsilon} \\ L \text{ dyadic}}} L^{1/2}(1 + M^{-1/4} L^{1/2})^2 \ll \frac{M^{3/4+\varepsilon}}{Z} Z^{1/2} \Big(1 + \frac{Z}{M^{1/2}}\Big).
\end{equation}
Simplifying this with $Z = N/X$ and $M = N^{2/3}$ leads to
\begin{equation}
 T_S''(N,X) \ll \frac{N^{2/3 + \varepsilon}}{X^{1/2}} + X^{1/2} N^{\varepsilon},
\end{equation}
as desired for \eqref{eq:smoothTsum}.
%
%
%
%
%
It is worthy of note that our final bound comes from a completely different (and much simpler) lattice point counting for integral points near the curve $4y^2 - 3xz$.  A similar feature also occurred in the work of Elkies, and the quadratic form can already be seen in \cite[(42)]{Elkies}.  However, there is a difference; in our situation, we have already extracted the expected number of points in \eqref{eq:Mainterm}, and we are summing integral points along the quadratic form with an oscillatory weight function.  In Elkies's case, the integral points along the quadratic form are providing the expected number of points itself.

%
%
%


\section{Speculation on improvements of Theorem \ref{thm:smoothTsum}}
\label{section:conjecture}
It would be of interest to extend Theorem \ref{thm:smoothTsum} to allow smaller values of $X$ compared to $N$.  Here we present some discussion on what such an extension might look like.  It is clear from \eqref{eq:TSonePoisson2} that if $X \gg N^{1+\varepsilon}$ (so $Z \ll X^{-\varepsilon}$), then $T_S(N,X) = \frac{XM}{N} \widehat{w_1}(0) \widehat{w_2}(0) + O(X^{-100})$, so the volume term is a very accurate count indeed.  However, as $X$ becomes smaller, in particular with $X \asymp N^{2/3}$, then the contribution from the Davenport family with $\delta = 1$ (i.e., \eqref{eq:Davenportexample}) might split off from the volume term.  This causes us to suspect the existence of a secondary main term hiding below the surface.

 The proof of Theorem \ref{thm:smoothTsum} gave
 \begin{equation}
  T_S(N,X) = \frac{XM}{N} \widehat{w_1}(0) \widehat{w_2}(0) + 2c \text{Re}(e^{-\pi i/4} T_S''(N,X))  + O(N^{-1} X^{1/2}),
 \end{equation}
where 
$T_S''(N,X)$ is given by \eqref{eq:TS''niceexpression}.  Writing $D = 4k_0^2 - 3lr$, and eliminating $r$ gives
\begin{equation}
\label{eq:TS''section3}
 T_S''(N, X) = \frac{3^{1/2} M^{3/4}}{2^{3/2} Z} \sum_{D \in \mz}  \sum_{l > 0} \frac{\widehat{w_2}\Big(\frac{l}{Z}\Big)}{l^{1/2}} W_1\Big(\frac{-D\sqrt{M}}{6l} \Big) S(D;3l),
\end{equation}
where
\begin{equation}
\label{eq:SDldef}
 S(D;3l) = \sum_{\substack{ x \shortmod{3l} \\ 4x^2 \equiv D \shortmod{3l}}} e\Big(\frac{4x^3 - 3D x}{27l^2} \Big).
\end{equation}
We remark that the sum over $x$ in \eqref{eq:SDldef} is well-defined modulo $3l$, which is a nice consistency check (we originally defined it via $0 \leq k_0 < 3l$). To see this, let 
$f(x) = 4x^3- 3Dx$, and note that $f(x+3ly) \equiv f(x) + 3ly f'(x) \pmod{27l^2}$, by a Taylor expansion, using the fact that $f''(x) \equiv f'''(x) \equiv 0 \pmod{3}$.   Furthermore, the condition $4x^2 \equiv D \pmod{3l}$ means that $f'(x) \equiv 0 \pmod{9l}$, so that $f(x+3ly) \equiv f(x) \pmod{27l^2}$ when $x$ is constrained to solutions to $4x^2 \equiv D \pmod{3l}$.

Since $W_1$ has rapid decay, we may practically assume $|D| \ll M^{-1/2} l N^{\varepsilon}$, so the sum over $D$ is much shorter than the sum over $l$.  In the crucial range $X = N^{2/3}$, we have $Z = N/X = N^{1/3}$, and $M^{1/2} = N^{1/3}$, so that the sum over $D$ is almost bounded in this range.  Therefore, we need to focus on the sum over $l$.  Based on both theoretical and computational evidence, we have
\begin{myconj}
\label{conj:sumofexpsums}
 Suppose that $D$ is not a square and $Y \geq 1$.  Then
 \begin{equation}
 \label{eq:SDsumconj}
 F(D;Y) := \sum_{l \leq Y} \frac{S(D;3l)}{\sqrt{l}} \ll Y^{\varepsilon},
 \end{equation}
 uniformly in $D \ll Y^{1+\varepsilon}$.  Furthermore,
\begin{equation}
\label{eq:T10conjecturalbound1}
 F(0;Y) = 3 Y^{1/2} + O(Y^{5/18 + \varepsilon}).
\end{equation}
\end{myconj}
In case $D \neq 0$ is a square, then we expect that the sum in \eqref{eq:SDsumconj} has a main term.  However, there are a number of annoying (but presumably surmountable) features that cause the calculation of the main term to be difficult.  In order to avoid these technical problems, yet capture some of the important features of $F(D;Y)$ for $D$ square, we study a modified version of $F$ as follows.
\begin{myconj}
\label{conj:sumofexpsums2}
 Let 
 \begin{equation}
  G(D;Y) = \sumstar_{\substack{l \leq Y \\ (l, 2D) =1}} \frac{1}{\sqrt{l}} \sum_{\substack{x \shortmod{l} \\ x^2 \equiv D \shortmod{l}}} e\Big(\frac{x^3 - 3D x}{l^2} \Big),
 \end{equation}
 where the star on the sum indicates that it is restricted to squarefree integers.
Then if $d \neq 0$, we have
\begin{equation}
\label{eq:sumofexpsums2}
 G(d^2;Y) \sim 2 \sumstar_{\substack{l \leq Y \\ (l, 2d) =1}} \frac{1}{\sqrt{l}},
\end{equation}
for $d$ fixed and $Y \rightarrow \infty$.
\end{myconj}

\subsection{Application to $T_S(N,X)$}
By partial summation, Conjecture \ref{conj:sumofexpsums} then implies that
\begin{equation}
\label{eq:TS''maintermsum}
 T_S''(N,X) = \frac{3^{1/2} M^{3/4}}{2^{3/2} Z} \sum_{d=0}^{\infty}  \sum_{l > 0} \frac{\widehat{w_2}\Big(\frac{l}{Z}\Big)}{l^{1/2}} W_1\Big(\frac{-d^2\sqrt{M}}{6l} \Big) S(d^2;3l) + O(N^{1/6 + \varepsilon}).
\end{equation}
In light of Conjecture \ref{conj:sumofexpsums2}, it seems reasonable to suppose that the sum over $l$ has an asymptotic, say of the form
\begin{equation}
 \sum_{l > 0} \frac{\widehat{w_2}\Big(\frac{l}{Z}\Big)}{l^{1/2}} W_1\Big(\frac{-d^2\sqrt{M}}{6l} \Big) S(d^2;3l) = c_d \sum_{l > 0} \frac{\widehat{w_2}\Big(\frac{l}{Z}\Big)}{l^{1/2}} W_1\Big(\frac{-d^2\sqrt{M}}{6l} \Big) + O(Z^{1/2-\delta}),
\end{equation}
for some constants $c_d > 0$ and $\delta > 0$.  However, it is difficult to go much further with this because the sum over $d$ may not have enough length for an asymptotic to emerge.  For instance, one may consider the terms with $l \asymp \sqrt{M}$ in which case the sum over $d$ is practically bounded.  Therefore, we choose to keep the sum over $d$ unsimplified, and so we derive
\begin{myconj}
 We have
 \begin{equation}
\label{eq:TS''conjecture}
 T_S''(N,X) = \frac{3^{1/2} M^{3/4}}{2^{3/2} Z} \sum_{d=0}^{\infty} c_d  \sum_{l > 0} \frac{\widehat{w_2}\Big(\frac{l}{Z}\Big)}{l^{1/2}} W_1\Big(\frac{-d^2\sqrt{M}}{6l} \Big)  + O\Big(\big(N^{1/6} +  \frac{N^{1/3}}{Z^{\delta}} + \frac{N^{1/2}}{Z^{13/18}}\big)N^{\varepsilon}\Big).
\end{equation}
\end{myconj}
The main point here is that the error term is non-trivial for $X \gg N^{2/3-\eta}$ for some fixed $\eta > 0$.

\begin{myquestion}
 Is it possible to see the contribution of \eqref{eq:Davenportexample} in the right hand side of \eqref{eq:TS''conjecture}?
\end{myquestion}

\subsection{Some justification for Conjecture \ref{conj:sumofexpsums}}
The theoretical evidence for this conjecture comes from the equidistribution of roots of quadratic congruences, which we briefly review.  Suppose that $f(x)$ is an irreducible quadratic polynomial  with integer coefficients, and for each $q$, consider the set of real numbers of the form $x/q$, where $x$ runs over solutions to $f(x) \equiv 0 \pmod{q}$.  Hooley \cite{Hooley} showed that the roots of $f$ are equidistributed modulo $1$, which by the Weyl criterion for equidistribution means that the 
Weyl sums defined by
\begin{equation}
 \rho_h(q) = \sum_{\substack{x \shortmod{q} \\ f(x) \equiv 0 \shortmod{q}}} e\Big(\frac{ h x}{q} \Big)
\end{equation}
satisfy, for $h \neq 0$,
\begin{equation}
 \sum_{q \leq Y} \rho_h(q) = o(Y).
\end{equation}
In fact, Hooley showed a power saving in $Y$.
For quadratic polynomials, it is even known that the roots to prime moduli are equidistributed \cite{DFI} \cite{Toth}.  
In light of this behavior, one might naturally guess that the solutions to $4x^2 \equiv D \pmod{3l}$, with $4x^2 - D$ irreducible (equivalently, $D$ is not a square), are ``random'', and so are the values of $4x^3 - 3Dx \pmod{27l^2}$, leading to cancellation in \eqref{eq:SDsumconj}.  It appears to be quite difficult to prove anything along these lines, however, because $S(D;3l)$ differs from $\rho_h(q)$ in two key ways.
Firstly, $4x^3 - 3Dx$ is much more oscillatory than the linear functions $hx$, and secondly, the exponential sum has a quite different shape since the denominator 
is a multiple of the modulus \emph{squared}.  Because of these difficulties, it seemed prudent to perform some numerical computations.  We calculated 
$F(D; Y)$ 
for a variety of choices of $D$ and $Y \leq 10^5$.  See 
Table \ref{table:FDYtable} for some values, where we have rounded the values of $F(D;Y)$.  
When $D$ is a square, the appearance of a main term is quite visible in the table.
By contrast, when $D$ is not a square, $F(D,Y)$ is substantially smaller. 
Our computations took long enough that it was not feasible to calculate $F(D;10^5)$ for many values of $D$, but see Figure \ref{fig:ScatterPlot} for a histogram of all the values of $F(D;10^4)$ for $|D| \leq 45000$, $D \neq \square$, $D \not \equiv 2 \pmod{3}$.  Note that $F(D;Y) = 0$ if $D \equiv 2 \pmod{3}$ since $S(D;3l) = 0$ if $D$ is not a square modulo $3$, which explains why we omitted these.
\begin{table}[h]
\caption{Values of $F(D;Y)$ for $D$ small}
\label{table:FDYtable}
\begin{tabular}{c|cccccccccc}
$ D $  & $-5$ & $-3$ & $-2$ & $0$ & $1$ & $3$ & $4$ & $6$ & $7$&  $9$ \\
\hline
$F(D;10^5) $ & $31.6$ & $9.7$ & $-2.2$ & $958.1$ & $654.8$ & $9.3$ & $1940.2$ & $.578$ & $22.5$  & $650.5$
\end{tabular}.
\end{table}
\begin{figure}[h]
\includegraphics{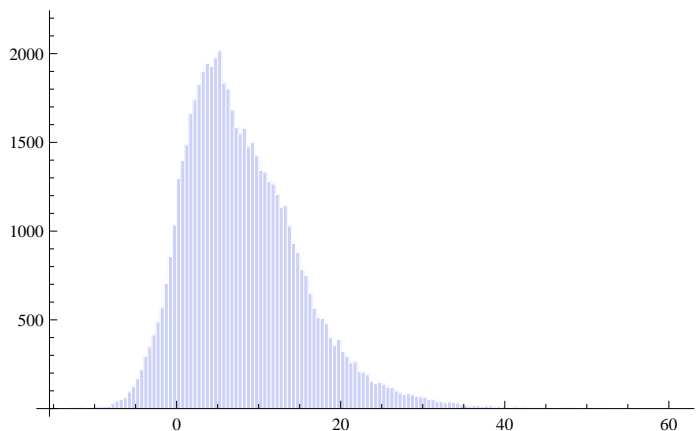}
\caption{Histogram of $F(D;10^4)$, $|D| \leq 45000$, $D \neq \square$, $D \not \equiv 2 \pmod{3}$.}
\label{fig:ScatterPlot}
\end{figure}
Figure \ref{fig:ScatterPlot} clearly indicates that $F(D;Y)$ has a strong tendency to be positive, indicating that it is unreasonable to expect any cancellation arising from a summation over $D$.  Figure \ref{fig:ScatterPlot} gives evidence towards the uniformity in $D$ assumed in Conjecture \ref{conj:sumofexpsums}.  Since the outliers are not visible in the figure, we record that for the listed values of $D$, $-15.3 < F(D;10^4) < 61.1$.  
We leave it as an interesting open problem to show that $F(D;Y)$ is positive on average over non-square $D$ (this would probably not directly improve our unconditional estimates of $T_S''(N,X)$ since the range of interest has $D$ essentially bounded).  This type of analysis could be helpful for understanding the true upper bounds in Conjecture \ref{conj:sumofexpsums}, assuming that the maximum growth (over $D$) is not too much larger than the average growth.

\subsection{The case $D=0$}
\label{section:D0}
Here we give evidence for \eqref{eq:T10conjecturalbound1}.
We begin by recording some obvious simplifications arising from setting $D=0$ in the definition of $F(D;Y)$, giving
\begin{equation}
 F(0;Y) = \sum_{l \leq Y} \frac{1}{\sqrt{l}} \sum_{1 \leq x \leq 3l} \sum_{\substack{r \in \mz \\ 3lr = 4x^2}} e\Big(\frac{4x^3}{27l^2}\Big).
\end{equation}
Write $(l,r) = g$, so $g | 2x$, and in addition $3(l/g)(r/g) = (2x/g)^2$.  Thus we may write $l= \delta_1 g q^2$, $r = \delta_2 g s^2$, where $\delta_1 \delta_2 = 3$, and $(q \delta_1, s \delta_2) = 1$.  Solving for $x$, we deduce $2x = 3g q s $ and the condition $1 \leq x \leq 3l$ translates to $1 \leq s \leq 2 \delta_1 q$.  Hence
\begin{equation}
 F(0;Y) = 
 \sum_{\delta_1 \delta_2 = 3}
\mathop{\sum_{\delta_1 g q^2 \leq Y}}_{\substack{2 | gsq, 1 \leq s \leq 2 \delta_1 q \\(q \delta_1, s \delta_2) = 1 }}
 \frac{1 }{(\delta_1 g q^2)^{1/2}}  e\Big(\frac{ g s^3}{2 \delta_1^2 q } \Big).
\end{equation}
 
We need to extract a main term.  Imagine that we have applied a dyadic partition of unity to the sum over $g$, and that we are considering the terms with $g \asymp R$.  Since the sum over $g$ is an exponential sum with linear phase, there will often be cancellation.  In case $2 | sq$, then $g$ has no constraint modulo $2$, and we see that the sum over $g$ is small if $R \gg q Y^{\varepsilon}$, except in the special case that $2 \delta_1^2 q |s^3$.  This divisibility implies $s=2$ by the following reasoning.  Since $(s, \delta_1 q) = 1$, the condition $2 \delta_1^2 q | s^3$ implies $\delta_1 = q = 1$, and hence $2 | s$.  Combining this with $1 \leq s \leq 2\delta_1 q = 2$, we obtain $s=2$.  If we assume $2 \nmid sq$, then $g$ runs over even integers, and the same argument as above shows the sum over $g$ is small if $R \gg q N^{\varepsilon}$, except if $\delta_1^2 q | s^3$.  As before, this means $\delta_1 = q = 1$, and so $s=1$.  

In summary, we have detected a main term that appears for $\delta_1 = q = 1$, and gives the following
\begin{equation}
F_0(0;Y) = \mathop{\sum_{ g \leq Y}}_{\substack{2 | g }}
 \frac{1 }{g^{1/2}} + \sum_{ g \leq Y}
 \frac{1 }{g^{1/2}} = 3 Y^{1/2} + O(1).
\end{equation}
Write $F_1(0;Y) = F(0;Y) - F_0(0;Y)$, so \eqref{eq:T10conjecturalbound1} amounts to $F_1(0;Y) \ll Y^{5/18 + \varepsilon}$.


For simplicity, let's consider the case $\delta_1 = 1$, $\delta_2 =3$, and $2 | g$.  The inner sum over $s$ is
\begin{equation}
H^*(g', q):= 2\sumstar_{s \shortmod{q}} e\Big(\frac{ g' s^3}{q } \Big),
\end{equation}
where $g/2 = g'$, and the star indicates $(s,q) = 1$.  
To remove the coprimality restriction, we use M\"{o}bius inversion, obtaining
\begin{equation}
H^*(g', q) = 2\sumstar_{s \shortmod{q}} e\Big(\frac{ g' s^3}{q } \Big) = 2\sum_{d | q} \mu(d)H(gd^2, q/d), \quad \text{where} \quad H(A,c) = \sum_{x \shortmod{c}} e\Big(\frac{A x^3}{c} \Big).
\end{equation}
Patterson \cite{PattersonHua1} \cite{PattersonHua2} has shown that if $A \neq 0$, then
\begin{equation}
\label{eq:PAX}
P(A;X) := \sum_{c \leq X} H(A,c) = k(A) X^{4/3} + O_{A,\varepsilon}(X^{5/4 + \varepsilon}),
\end{equation}
for a certain explicit constant $k(A)$.  Patterson's approach is to relate the generating Dirichlet series $\sum_{c=1}^{\infty} \frac{H(A,c)}{c^s}$ to sums of cubic Gauss sums.  Then the theory of metaplectic Eisenstein series leads to bounds on these latter sums.  In these arguments, it is difficult to track the dependence of $P(A;X)$ on $A$.  For some work with a similar flavor, see Louvel \cite{Louvel}, who studied sums like $P(A;X)$ but with $c$ running over Eisenstein integers in an arithmetic progression (and with $A=1$).  
Here $k(A)$ is a multiplicative function of size $\approx A^{-1/3 + o(1)}$, provided that $A$ is cube-free.  Based in part on some numerical calculations, we are led to conjecture that
\begin{equation}
P(A;X) \ll X^{4/3+\varepsilon},
\end{equation}
uniformly for $1 \leq |A| \ll X^{1+\varepsilon}$.  See Figure \ref{fig:ScatterPlot2} for a histogram of $P(A;X)/X^{4/3}$, with $X = 10^4$, and $1 \leq A \leq 6000$ (it suffices to consider $A > 0$ since $H(-A, c) = H(A,c)$).
\begin{figure}[h]
\includegraphics{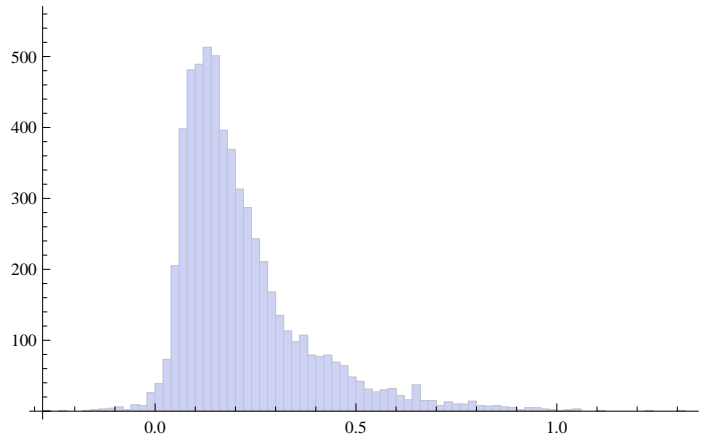}
\caption{Histogram of $P(A;X)/X^{4/3}$, with $X = 10^4$, and $1 \leq |A| \leq 6000$}
\label{fig:ScatterPlot2}
\end{figure}
These numerical experiments indicate that the error term in \eqref{eq:PAX} does not have a large power of $A$;  perhaps it is $O(|A|^{\varepsilon} X^{5/4+\varepsilon})$ (or even a smaller power of $X$)?  
As Patterson notes (see \cite[p.727]{PattersonHua2}), the numerical agreement in \eqref{eq:PAX} varies significantly with $X$ making it difficult to predict the size of the true error term in \eqref{eq:PAX}, even for $A$ fixed (not to mention with some uniformity in $A$).
Furthermore, for such ranges of $A$, these sums of exponential sums are almost always positive (this is clearly illustrated in Figure \ref{fig:ScatterPlot2}), so there is no more room for any cancellation.  As an aside, we have that for all $1\leq A \leq 6000$, and $X = 10^4$, that $-0.28 < P(A,X)/X^{4/3} < 1.31$.

Now we return to bounding $F_1(0;Y)$.  Consider the terms with $g \asymp R$, $q \asymp Q$.  By previous reasoning, we already know the terms with $R \gg QY^{\varepsilon}$ should be negligible, so we may assume $R \ll Q Y^{\varepsilon}$. 
We are thus led to the conjectured bound
\begin{equation}
 F_1(0;Y) \ll Y^{\varepsilon} \max_{\substack{Q^2 R \ll Y N^{\varepsilon} \\ R \ll Q Y^{\varepsilon} }} \frac{Q^{4/3} R}{Q \sqrt{R}} \ll Y^{\varepsilon} \max_{1 \ll Q \ll Y^{1/2}} Q^{1/3} \min\Big(\frac{Y^{1/2}}{Q}, Q^{1/2} \Big) \ll Y^{\frac{5}{18} + \varepsilon},
\end{equation}
as claimed.

\subsection{Evidence for Conjecture \ref{conj:sumofexpsums2}}
For this, we remark that the solutions to $x^2 \equiv d^2 \pmod{l}$ correspond precisely to factorizations $l = l_1 l_2$, where $x \equiv d \pmod{l_1}$ and $x \equiv -d \pmod{l_2}$.  To see this, note that if $p | l$ then $x^2 \equiv d^2 \pmod{p}$ means that $x \equiv \pm d \pmod{p}$, which are distinct residue classes if $p \nmid 2d$.  Therefore, since $l$ is squarefree and $(l, 2d) = 1$, the claimed result follows using the Chinese remainder theorem.  Using the Chinese remainder theorem again, we then have
\begin{equation}
 G(d^2;Y) = \sumstar_{\substack{l_1 l_2 \leq Y \\ (l_1 l_2, 2d) =1 \\ (l_1, l_2) = 1}} \frac{1}{\sqrt{l_1 l_2}} \sum_{\substack{x_1 \shortmod{l_1} \\ x_1 \equiv d \shortmod{l_1}}} e\Big(\frac{\overline{l_2^2}(x_1^3 - 3d^2 x_1)}{l_1^2} \Big)  \sum_{\substack{x_2 \shortmod{l_2} \\ x_2 \equiv - d \shortmod{l_2}}} e\Big(\frac{\overline{l_1^2}(x_2^3 - 3d^2 x_2)}{l_2^2} \Big),
\end{equation}
which quickly simplifies as
\begin{equation}
 G(d^2;Y) = \sumstar_{\substack{l_1 l_2 \leq Y \\ (l_1 l_2, 2d) =1 \\ (l_1, l_2) = 1}} \frac{1}{\sqrt{l_1 l_2}}  e\Big(2d^3 \Big(\frac{\overline{l_1^2}}{l_2^2}- \frac{\overline{l_2^2}}{l_1^2} \Big)\Big).
\end{equation}
An astute reader may notice the appearance of a Sali\'{e} sum to modulus $(l_1 l_2)^2$.  For instance, see (12.43) of \cite{IK}.  However, this is not directly helpful for bounding $G$ because the moduli run over \emph{squares} (with some minor congruence conditions), while the known results on sums of Sali\'{e} sums run over all moduli (again, one may allow congruence conditions too).

Recall the reciprocity law $\frac{\overline{a}}{b} \equiv - \frac{\overline{b}}{a} + \frac{1}{ab} \pmod{1}$.  We use this with $a = l_1^2$, $b= l_2^2$ in case $l_1 < l_2$.  In the opposite situation we switch the roles of $a$ and $b$, so that in all cases we are reducing the modulus.  Thus we arrive at a sum of the form
\begin{equation}
 \sum_{l_1} \sum_{l_2 > l_1} \frac{e\big(\frac{2d^3}{(l_1 l_2)^2}\big)}{\sqrt{l_1 l_2}} e\Big(-4d^3 \frac{\overline{l_2^2}}{l_1^2} \Big),
\end{equation}
as well as a similar one with $l_1 \geq l_2$.  We predict that the sum over $l_2$ has some cancellation except when $l_1 = 1$.  One way to see this is that the completed sum vanishes, that is
\begin{equation}
\label{eq:completedsum}
 \sumstar_{\alpha \shortmod{q^2}} e\Big(\frac{\beta \overline{\alpha}^2}{q^2}\Big) = 0,
\end{equation}
if $(2\beta, q) = 1$.  In turn, a way to verify \eqref{eq:completedsum} is to write $\alpha = \alpha_0(1+\alpha_1 q)$ where $\alpha_0$ runs over $(\mz/q\mz)^*$, and $\alpha_1$ runs over $\mz/q\mz$.  Then $\overline{\alpha} = \overline{\alpha_0}(1-\alpha_1 q)$ (where $\overline{\alpha_0}$ is some integer satisfying $\alpha_0 \overline{\alpha_0} \equiv 1 \pmod{q^2}$), and so $\overline{\alpha}^2 \equiv \overline{\alpha_0}^2(1-2\alpha_1 q) \pmod{q^2}$.  The sum over $\alpha_1$ then vanishes.  Taking into account the other case $l_2 = 1$, we obtain the conjecture
\begin{equation}
 G(d^2;Y) = \sumstar_{\substack{l \leq Y \\ (l, 2d) =1}} \frac{2 \cos\Big(\frac{4 \pi d^3}{l^2}\Big)}{\sqrt{l}} + o(\sqrt{Y}).
\end{equation}
Since $d$ is held fixed, the Taylor expansion for cosine leads to \eqref{eq:sumofexpsums2}.


\begin{thebibliography}{99}
  \bibitem[BB]{BaierBrowning} S. Baier and T. Browning, \emph{
Inhomogeneous cubic congruences and rational points on del Pezzo surfaces.} 
J. Reine Angew. Math. 680 (2013), 69--151. 
  \bibitem[BeSt]{BS} F. Beukers and C. L. Stewart,
\emph{Neighboring powers.}
J. Number Theory 130 (2010), no. 3, 660--679. 
  \bibitem[BCMS]{BCMS} B. J. Birch, S. Chowla, M. Hall, A. Schinzel, \emph{
On the difference $x^3- y^2$}.
Norske Vid. Selsk. Forh. (Trondheim) 38 1965 65--69. 
\bibitem[BKY]{BKY} V. Blomer, R. Khan, M. Young, \emph{Distribution of Mass of Hecke eigenforms}.  Duke Math. J. 162 (2013), no. 14, 2609--2644.
  \bibitem[BP]{BombieriPila} E. Bombieri and J. Pila, \emph{The number of integral points on arcs and ovals.}
Duke Math. J. 59 (1989), no. 2, 337--357. 
  \bibitem[BMcG]{BrumerMcGuinness} A. Brumer and O. McGuinness, \emph{The behavior of the Mordell-Weil group of elliptic curves.}
Bull. Amer. Math. Soc. (N.S.) 23 (1990), no. 2, 375--382. 
 \bibitem[BrSi]{BrumerSilverman} A. Brumer and J. Silverman, \emph{The number of elliptic curves over {$\mathbf Q$} with conductor
              {$N$}}.
Manuscripta Math. 91 (1996), no. 1, 95--102. 
\bibitem[Da]{Davenport} H. Davenport, \emph{On $f^3(t)-g^2(t)$}. Norske Vid. Selsk. Forh. (Trondheim) 38 (1965), 86--87.
\bibitem[Du]{Dujella} A. Dujella, \emph{On Hall's conjecture.}
Acta Arith. 147 (2011), no. 4, 397--402. 
\bibitem[DFI]{DFI}  W. Duke, J.B. Friedlander, and H. Iwaniec, \emph{ Equidistribution of roots of a quadratic congruence to prime moduli.} Ann. of Math. (2) 141 (1995), no. 2, 423--441.

  \bibitem[DK]{DukeKowalski} W. Duke and E. Kowalski, \emph{A problem of Linnik for elliptic curves and mean-value estimates for automorphic representations.}
With an appendix by Dinakar Ramakrishnan.
Invent. Math. 139 (2000), no. 1, 1--39. 
\bibitem[E]{Elkies} N. Elkies, \emph{Rational points near curves and small nonzero $|x^3-y^2|$ via lattice reduction}. Algorithmic number theory (Leiden, 2000), 33--63, Lecture Notes in Comput. Sci., 1838, Springer, Berlin, 2000.
\bibitem[EV]{EllenbergVenkatesh} J. Ellenberg and A. Venkatesh, \emph{
Reflection principles and bounds for class group torsion.}
Int. Math. Res. Not. IMRN 2007, no. 1, Art. ID rnm002, 18 pp. 
  \bibitem[FNT]{FNT} E. Fouvry, M. Nair, and G. Tenenbaum, \emph{
L'ensemble exceptionnel dans la conjecture de Szpiro.} 
Bull. Soc. Math. France 120 (1992), no. 4, 485--506. 
      \bibitem[Ha]{Hall} M. Hall, \emph{The Diophantine equation $x^3 - y^2 =k$.}  
Computers in number theory (Proc. Sci. Res. Council Atlas Sympos. No. 2, Oxford, 1969), pp. 173--198. Academic Press, London, 1971.
\bibitem[HV]{HelfgottVenkatesh} H. Helfgott and A. Venkatesh, \emph{
Integral points on elliptic curves and $3$-torsion in class groups.}
J. Amer. Math. Soc. 19 (2006), no. 3, 527--550. 
\bibitem[Ho]{Hooley} C. Hooley, \emph{On the number of divisors of a quadratic polynomial.}
Acta Math. 110 1963 97--114. 
      \bibitem[Hu]{Huxley} M. Huxley, \emph{The integer points in a plane curve}.  Funct. Approx. Comment. Math. 37 (2007), part 1, 213--231.
  \bibitem[HT]{HuxleyTrifonov} M. Huxley and O. Trifonov, \emph{
The square-full numbers in an interval.}
Math. Proc. Cambridge Philos. Soc. 119 (1996), no. 2, 201--208.
\bibitem[IK]{IK} H. Iwaniec and E. Kowalski, {\em Analytic number theory.\/}
American Mathematical Society Colloquium Publications, 53. American Mathematical
Society, Providence, RI, 2004. xii+615 pp.
  \bibitem[La]{Lang} S. Lang, \emph{Old and new conjectured Diophantine inequalities.}
Bull. Amer. Math. Soc. (N.S.) 23 (1990), no. 1, 37--75.
\bibitem[Lo]{Louvel} B. Louvel, \emph{On the distribution of cubic exponential sums}. To appear in Forum Mathematicum.
\bibitem[Pa1]{PattersonHua1} S. J. Patterson, \emph{On the distribution of certain Hua sums.} 
Asian J. Math. 4 (2000), no. 4, 977--985.  
\bibitem[Pa2]{PattersonHua2} S. J. Patterson, \emph{On the distribution of certain Hua sums. II.} Asian J. Math. 6 (2002), no. 4, 719--729.
\bibitem[Pi]{Pierce} L. Pierce, \emph{The $3$-part of class numbers of quadratic fields.}.
J. London Math. Soc. (2) 71 (2005), no. 3, 579--598. 
  \bibitem[S-D]{S-D} H. P. F. Swinnerton-Dyer, \emph{The number of lattice points on a convex curve.} J. Number Theory 6 (1974), 128--135.
  \bibitem[To]{Toth} {\'A}. T{\'o}th, 
\emph{Roots of quadratic congruences.}
Internat. Math. Res. Notices 2000, no. 14, 719--739. 
  \bibitem[Tr]{Trifonov} O. Trifonov, \emph{The integer points close to a smooth curve}.  Serdica Math. J. 24 (1998), no. 3-4, 319--338.
  \bibitem[W]{Watkins} M. Watkins, \emph{Some heuristics about elliptic curves.} Experiment. Math. 17 (2008), no. 1, 105--125.
  \bibitem[Z]{Zannier} U. Zannier, \emph{On {D}avenport's bound for the degree of {$f^3-g^2$} and
              {R}iemann's existence theorem.} Acta Arith. 71 (1995), no. 2, 107--137. 
 \end{thebibliography}
\end{document}